\documentclass[a4paper,twoside,11pt]{article}
\usepackage{a4wide,graphicx,fancyhdr,amsmath,amssymb,mathtools,yfonts}
\usepackage[all]{xy}
\usepackage[utf8]{inputenc}
\usepackage{amsthm}
\usepackage[english]{babel}
\usepackage{chngcntr}
\usepackage{ifthen}
\usepackage{calc}
\usepackage{hyperref}

\counterwithin*{equation}{section}
\counterwithin*{equation}{subsection}

\fancypagestyle{plain}{%
\fancyhf{}
\fancyhead[LO,RE]{\sffamily\bfseries\large Leiden University}
\fancyhead[RO,LE]{\sffamily\bfseries\large Research}
\fancyfoot[LO,RE]{\sffamily\bfseries\large Mathematical Institute}
\fancyfoot[RO,LE]{\sffamily\bfseries\thepage}

}

\newtheorem{theorem}{Theorem}
\newtheorem{lemma}[theorem]{Lemma}

\newtheorem{corollary}[theorem]{Corollary}

\newcounter{Constant}[chapter]
\newcounter{constant}[chapter]

\newcommand{\Cst}[1]
{
\ifthenelse{\value{#1}=0}
  {\addtocounter{Constant}{1}\setcounter{#1}{\value{Constant}}f_{\arabic{#1}}}
  {f_{\arabic{#1}}}
}

\newcommand{\cst}[1]
{
\ifthenelse{\value{#1}=0}
  {\addtocounter{constant}{1}\setcounter{#1}{\value{constant}}c_{\arabic{#1}}}
  {c_{\arabic{#1}}}
}

\title{\vspace{-\baselineskip}\sffamily\bfseries On the equation $x + y = 1$ in finitely generated groups in positive characteristic}
\author{\huge{Peter Koymans, Carlo Pagano} \\
\\{\tt p.h.koymans@math.leidenuniv.nl}
\\{\tt carlein90@gmail.com}}

\date{\today}

\begin{document}
\maketitle

\section{Introduction}
Let $G$ be a subgroup of $\mathbb{C}^\ast \times \mathbb{C}^\ast$ with coordinatewise multiplication. Assume that the rank $\text{dim}_\mathbb{Q} \ G \otimes_\mathbb{Z} \mathbb{Q} = r$ is finite. Beukers and Schlickewei \cite{BS} proved that the equation
\[
x + y = 1
\]
in $(x, y) \in G$ has at most $2^{8r + 8}$ solutions. A key feature of their upper bound is that it depends only on $r$. 

In this paper we will analyze the characteristic $p$ case. To be more precise, let $p > 0$ be a prime number and let $K$ be a field of characteristic $p$. Let $G$ be a subgroup of $K^\ast \times K^\ast$ with $\text{dim}_\mathbb{Q} \ G \otimes_\mathbb{Z} \mathbb{Q} = r$ finite. Then Voloch proved in \cite{Voloch} that an equation
\[
ax + by = 1 \text{ in } (x, y) \in G
\]
for given $a, b \in K^\ast$ has at most $p^r(p^r + p - 2)/(p - 1)$ solutions $(x, y) \in G$, unless $(a, b)^n \in G$ for some $n \geq 1$. 

Voloch also conjectured that this upper bound can be replaced by one depending only on $r$. Our main theorem answers this conjecture positively.

\begin{theorem}
\label{tVoloch}
Let $K$, $G$, $r$, $a$ and $b$ be as above. Then the equation
\begin{align}
\label{eVoloch}
ax + by = 1
\end{align}
in $(x, y) \in G$ has at most $31 \cdot 19^{r + 1}$ solutions $(x, y)$ unless $(a, b)^n \in G$ for some $n \geq 1$ with $(n, p) = 1$.
\end{theorem}

Our main theorem will be a consequence of the following theorem.

\begin{theorem}
\label{tXY}
Let $K$ be a field of characteristic $p > 0$ and let $G$ be a finitely generated subgroup of $K^\ast \times K^\ast$ of rank $r$. Then the equation
\begin{align}
\label{eXY}
x + y = 1 \text{ in } (x, y) \in G
\end{align}
has at most $31 \cdot 19^{r}$ solutions $(x, y)$ satisfying $(x, y) \not \in G^p$.
\end{theorem}

Clearly, the last condition is necessary to guarantee finiteness. Indeed if we have any solution to $x + y = 1$, then we get infinitely many solutions $x^{p^k} + y^{p^k} = 1$ for $k \in \mathbb{Z}_{\geq 0}$ due to the Frobenius operator.

The set-up of the paper is as follows. We start by introducing the basic theory about valuations that is needed for our proofs. Then we derive Theorem \ref{tXY} by generalizing the proof of Beukers and Schlickewei \cite{BS} to positive characteristic. We remark that their proof heavily relies on techniques from diophantine approximation. Most of the methods from diophantine approximation can not be transferred to positive characteristic, so that this is possible with the method of Beukers and Schlickewei is a surprising feat on its own. It will be more convenient for us to follow \cite{EG}, which is directly based on the proof of Beukers and Schlickewei. Theorem \ref{tVoloch} will be a simple consequence of Theorem \ref{tXY}.

\section{Valuations and heights}
Our goal in this section is to recall the basic theory about valuations and heights without proofs. To prove Theorem \ref{tXY} we may assume without loss of generality that $K = \mathbb{F}_p(G)$. Thus, $K$ is finitely generated over $\mathbb{F}_p$. Note that Theorem \ref{tXY} is trivial if $K$ is algebraic over $\mathbb{F}_p$, so from now on we further assume that $K$ has positive transcendence degree over $\mathbb{F}_p$. The algebraic closure of $\mathbb{F}_p$ in $K$ is a finite field, which we denote by $\mathbb{F}_q$. Then there is an absolutely irreducible, normal projective variety $V$ defined over $\mathbb{F}_q$ such that its function field $\mathbb{F}_q(V)$ is isomorphic to $K$.

Fix a projective embedding of $V$ such that $V \subseteq \mathbb{P}^M_{\mathbb{F}_q}$ for some positive integer $M$. A prime divisor $\mathfrak{p}$ of $V$ over $\mathbb{F}_q$ is by definition an irreducible subvariety of $V$ of codimension one. Recall that for a prime divisor $\mathfrak{p}$ the local ring $\mathcal{O}_\mathfrak{p}$ is a discrete valuation ring, since $V$ is non-singular in codimension one. Following \cite{L2} we will define heights on $V$. To do this, we start by defining a set of normalized discrete valuations
\[
M_K := \{\text{ord}_\mathfrak{p} : \mathfrak{p} \text{ prime divisor of } V\},
\]
where $\text{ord}_\mathfrak{p}$ is the normalized discrete valuation of $K$ corresponding to $\mathcal{O}_\mathfrak{p}$. If $v = \text{ord}_\mathfrak{p} \in M_K$, we define for convenience $\deg v := \deg \mathfrak{p}$ with $\deg \mathfrak{p}$ being the projective degree in $\mathbb{P}^M_{\mathbb{F}_q}$. Then the set $M_K$ satisfies the sum formula
\[
\sum_{v \in M_K} v(x) \deg v = 0
\]
for $x \in K^\ast$. This is indeed a well-defined sum, since for $x \in K^\ast$ there are only finitely many valuations $v$ satisfying $v(x) \neq 0$. Furthermore, we have $v(x) = 0$ for all $v \in M_K$ if and only if $x \in \mathbb{F}_q^\ast$. If $P$ is a point in $\mathbb{A}^{n + 1}(K) \setminus \{0\}$ with coordinates $(y_0, \ldots, y_n)$ in $K$, then its homogeneous height is
\[
H_K^{\text{hom}}(P) = -\sum_{v \in M_K} \min_i \{v(y_i)\} \deg v
\]
and its height
\[
H_K(P) = H_K^{\text{hom}}(1, y_0, \ldots, y_n).
\]
We will need the following properties of the height.

\begin{lemma}
\label{lHeight}
Let $P \in \mathbb{A}^{n + 1}(K) \setminus \{0\}$. The height defined above has the following properties: \\
1) $H_K^{\text{hom}}(\lambda P) = H_K^{\text{hom}}(P)$ for $\lambda \in K^\ast$. \\
2) $H_K^{\text{hom}}(P) \geq 0$ with equality if and only if $P \in \mathbb{P}^n(\mathbb{F}_q)$.
\end{lemma}


\section{Proof of Theorem \ref{tXY}}
This section is devoted to the proof of Theorem \ref{tXY}. We will follow the proof in \cite{EG}, see Section 6.4, with some crucial modifications to take care of the presence of the Frobenius map. Let us start with a simple lemma.

\begin{lemma}
\label{boundpr}
The equation
\begin{align}
\label{eXY2}
x + y = 1 \text{ in } (x, y) \in G
\end{align}
has at most $p^r$ solutions $(x, y)$ satisfying $x \not \in K^p$ and $y \not \in K^p$.
\end{lemma}

\begin{proof}
Let $u = (u_1, u_2)$ and $v = (v_1, v_2)$ be two solutions of (\ref{eXY2}). We claim that $u \equiv v \mod G^p$ implies $u = v$. Indeed, if $u \equiv v \mod G^p$, we can write $v_1 = u_1 \gamma^p$ and $v_2 = u_2 \delta^p$ with $(\gamma, \delta) \in G$. In matrix form this means that
\[
\begin{pmatrix}
1 & 1 \\
\gamma^p & \delta^p
\end{pmatrix}
\begin{pmatrix}
u_1 \\
u_2
\end{pmatrix}
=
\begin{pmatrix}
1 \\
1
\end{pmatrix}
.
\]
For convenience we define
\[
A :=
\begin{pmatrix}
1 & 1 \\
\gamma^p & \delta^p
\end{pmatrix}
.
\]
If $A$ is invertible, we find that $u_1, u_2 \in K^p$ contrary to our assumptions. So $A$ is not invertible, which implies that $\gamma = \delta = 1$. This proves the claim.

The claim implies that the number of solutions is at most $|G/G^p|$. Let $\mathbb{F}_q$ be the algebraic closure of $\mathbb{F}_p$ in $G$. It is a finite extension of $\mathbb{F}_p$, since $G$ is finitely generated over $\mathbb{F}_p$. It follows that $G^{\text{tors}} \subseteq \mathbb{F}_q^\ast \times \mathbb{F}_q^\ast$. Hence $|G^{\text{tors}}| \mid (q - 1)^2$, which is co-prime to $p$. We conclude that $|G/G^p| = p^r$ as desired.
\end{proof}

Lemma \ref{boundpr} gives the following corollary.

\begin{corollary}
\label{boundpr2}
The equation
\begin{align}
\label{eXY3}
x + y = 1 \text{ in } (x, y) \in G
\end{align}
has at most $p^r$ solutions $(x, y)$ satisfying $(x, y) \not \in G^p$.
\end{corollary}

\begin{proof}
Define
\[
G' := \{(x, y) \in K \times K : (x^N, y^N) \in G \text{ for some } N \in \mathbb{Z}_{>0}\}.
\]
It is a well known fact that $G'$ is finitely generated if $G$ and $K$ are. It follows that $G'$ is a finitely generated group of rank $r$. To complete the proof we will give an injective map from the solutions $(x, y) \in G$ of (\ref{eXY3}) satisfying $(x, y) \not \in G^p$ to the solutions $(x', y') \in G'$ of (\ref{eXY3}) satisfying $(x', y') \not \in G'^p$.

So let $(x, y) \in G$ be a solution of (\ref{eXY3}). We remark that $x, y \not \in \mathbb{F}_q$. Hence we can repeatedly take $p$-th roots until we get $x', y' \not \in K^p$. Using heights one can prove that this indeed stops after finitely many steps. Then it is easily verified that $(x', y') \in G'$ is a solution of (\ref{eXY3}) and that the map thus defined is injective. Now apply Lemma \ref{boundpr}.
\end{proof}

By Corollary \ref{boundpr2} we may assume that $p$ is sufficiently large throughout, say $p > 7$. Both the proof in \cite{EG} and our proof rely on very special properties of the family of binary forms $\{W_N(X,Y)\}_{N \in \mathbb{Z}_{>0}}$ defined by the formula 
\[
W_{N}(X, Y) = \sum_{m = 0}^{N} \binom{2N - m}{N - m} \binom{N + m}{m} X^{N - m} (-Y)^{m}.
\] 
We have for all positive integers $N$ that $W_N(X,Y) \in \mathbb{Z}[X,Y]$. Furthermore, setting $Z = -X - Y$, the following statements hold in $\mathbb{Z}[X,Y]$.

\begin{lemma}
\label{magic polynomials}
1) $W_N(Y, X) = (-1)^N W_N(X, Y)$. \\
2) $X^{2N + 1} W_{N}(Y, Z) + Y^{2N + 1} W_N(Z, X) + Z^{2N + 1} W_N(X,Y) = 0$. \\
3) There exist a non-zero integer $c_N$ such that 
\[
\text{det}\begin{pmatrix}
Z^{2N + 1} W_N(X, Y) & Y^{2N + 1} W_N(Z, X)  \\
Z^{2N + 3} W_{N + 1}(X, Y) & Y^{2N + 3} W_{N + 1}(Z, X) \\
\end{pmatrix} 
= c_N(XYZ)^{2N+1}(X^2+XY+Y^2).
\]
\end{lemma}

\begin{proof}
This is Lemma 6.4.2 in \cite{EG}.
\end{proof}

Since the formulas in the previous lemma hold in $\mathbb{Z}[X,Y]$ they hold in every field $K$. But if $\text{char}(K) = p > 0$ and $p \mid c_N$, then part 3) of Lemma \ref{magic polynomials} tells us that 
\[
det
\begin{pmatrix}
Z^{2N + 1} W_N(X,Y) & Y^{2N + 1} W_N(Z,X)  \\ 
Z^{2N + 3} W_{N+1}(X,Y) & Y^{2N + 3} W_{N+1}(Z,X) \\ 
\end{pmatrix}
= 0
\]
in $K[X,Y]$. The following remarkable identity will be handy later on, when we need that $c_N$ does not vanish modulo $p$.

\begin{lemma} 
\label{magic identity}
For every positive integer $N$, one has $W_N(2,-1) = 4^{N} \binom{\frac{3}{2}N}{N}$.
\end{lemma} 

\begin{proof}
It is enough to evaluate $\sum_{i=0}^{N}  \binom{2N - i}{N} \binom{N + i}{N} 2^{-i}$. We have
\[
\sum_{i = 0}^{N} \binom{2N - i}{N} \binom{N + i}{N} 2^{-i} = \binom{2N}{N} F\left(-N, N + 1, -2N, \frac{1}{2}\right),
\] 
where $F(a,b,c,z)$ is the hypergeometric function defined by the power series $F(a,b,c,z) := \sum_{i = 0}^{\infty} \frac{(a)_i (b)_i}{i! (c)_i}z^{n}$. Here we define for a real $t$ and a non-negative integer $i$ $(t)_i = 1$ if $i = 0$ and for $i$ positive $(t)_i = t (t + 1) \cdot \ldots \cdot (t + i - 1)$. Now the desired result follows from Bailey's formulas where special values of the function $F$ are expressed in terms of values of the $\Gamma$-function, see \cite{Bailey} page 297.
\end{proof}

We obtain the following corollary.

\begin{corollary} 
\label{range of safeness}
Let $p$ be an odd prime number and let $N$ be a positive integer with $N < \frac{p}{3} - 2$. Then $c_N \not \equiv 0 \ \text{mod} \ p$.
\end{corollary}

\begin{proof}
Indeed one has that 
\[
det
\begin{pmatrix}
Z^{2N + 1} W_N(X, Y) & Y^{2N + 1} W_N(Z, X) \\
Z^{2N + 3} W_{N+1}(X, Y) & Y^{2N + 3} W_{N + 1}(Z, X) \\
\end{pmatrix}
\]
evaluated at $(X, Y, Z) = (2, -1, -1)$ gives up to sign $2W_N(2,-1)W_{N+1}(2,-1)$. By the previous proposition, this is a power of $2$ times the product of two binomial coefficients whose top terms are less than $p$, hence it can not be divisible by $p$.
\end{proof}

We now state and prove the analogues of Lemmata 6.4.3-6.4.5 from \cite{EG} for function fields of positive characteristic. 

\begin{lemma}
\label{exterior product}
Let $a, b, c$ be non-zero elements of $K$, and let $(x_i, y_i, z_i)$ for $i = 1, 2$ be two $K$-linearly independent vectors from $K^3$ such that $ax_i + by_i + cz_i=0$ for $i = 1, 2$. Then 
\[
H_K^{\text{hom}}(a, b, c) \leq H_K^{\text{hom}}(x_1, y_1, z_1) + H_K^{\text{hom}}(x_2, y_2, z_2).
\]
\end{lemma}

\begin{proof}
The vector $(a, b, c)$ is $K$-proportional to the vector $(y_1z_2 - y_2z_1, z_1x_2 - x_1z_2, x_1y_2 - x_2y_1)$. So we have 
\begin{align*}
H_K^{\text{hom}}(a, b, c) &= H_K^{\text{hom}}(y_1z_2 - y_2z_1,z_1x_2 - x_1z_2,x_1y_2 - x_2y_1) \\
&=\sum_{v \in M_K} -\text{min}(v(y_1z_2 - y_2z_1),v(z_1x_2 - x_1z_2),v(x_1y_2 - x_2y_1)) \deg v \\
&\leq \sum_{v \in M_K} -\text{min}(v(y_1),v(z_1),v(x_1)) \deg v + \sum_{v \in M_K} -\text{min}(v(z_2),v(x_2),v(y_2)) \deg v \\
&= H_K^{\text{hom}}(x_1, y_1, z_1) + H_K^{\text{hom}}(x_2, y_2, z_2),
\end{align*}
which was the claimed inequality. 
\end{proof}

We apply Lemma \ref{exterior product} to the unit equation.

\begin{lemma} 
\label{difference between solutions}
Let $u = (u_1, u_2), v = (v_1, v_2)$ be two solutions of (\ref{eXY}) with $u \neq v$. Then we have $H_K(u) \leq H_K(vu^{-1})$.
\end{lemma}

\begin{proof}
Apply Lemma \ref{exterior product} with $(a, b, c) = (u_1, u_2, -1)$, $(x_1, y_1, z_1) = (1, 1, 1)$, $(x_2, y_2, z_2) = (v_1u_1^{-1}, v_2u_2^{-1}, 1)$ and use the fact that $H_K^{\text{hom}}(1,1,1) = 0$.
\end{proof}

The next Lemma takes advantage of the properties of $W_N(X, Y)$ listed in Lemma \ref{magic polynomials} and the non-vanishing of $c_N$ modulo $p$ obtained in Corollary \ref{range of safeness}.

\begin{lemma} 
\label{difference between solutions and multiples of solutions}
Let $u, v$ be as in Lemma \ref{difference between solutions}. Let $N < \frac{p}{3}-2$. Then there exists $M \in \{N, N+1\}$ such that $H_K(u) \leq \frac{1}{M + 1} H_K(vu^{-2M - 1})$.
\end{lemma}

\begin{proof}
The proof is almost the same as in Lemma 6.4.5 in \cite{EG}, with only few necessary modifications. For completeness we give the full proof.

If $u_1$, and thus both $u_1$ and $u_2$ are roots of unity, we have that $H_K(u)=0$ so the lemma is trivially true. By Lemma \ref{magic polynomials} part 2) we get that for $M \in \{N, N + 1\}$ the following holds:
\[
{u_1}^{2M + 1} W_{M}(u_2,-1) + {u_2}^{2M + 1} W_M(-1, u_1) - W_M(u_1, u_2) = 0
\]
as well as 
\[
{u_1}^{2M + 1}(v_1{u_1}^{-2M - 1}) + {u_2}^{2M + 1}(v_2{u_2}^{-2M - 1}) - 1 = 0.
\]
Now we claim that there is $M \in \{N,N+1\}$ such that the vectors
\begin{align}
\label{eLin}
(v_1, v_2, -1) \text{ and } ({u_1}^{2M + 1} W_{M}(u_2,-1), {u_2}^{2M+1} W_M(-1, u_1), -W_M(u_1, u_2))
\end{align}
are linearly independent. Clearly, to prove the claim it is enough to prove that the two vectors 
\[
({u_1}^{2M + 1} W_{M}(u_2, -1), {u_2}^{2M + 1} W_M(-1, u_1), -W_M(u_1, u_2)) \quad (M \in \{N, N + 1\})
\]
are linearly independent. But we know that for $M \in \{N, N + 1\}$ we have that $c_M \not \equiv 0 \ \text{mod} \ p$ by Corollary \ref{range of safeness} and the assumption that $N < \frac{p}{3}-2$. Furthermore, $u_1$ and $u_2$ are not algebraic over $\mathbb{F}_p$. Thus the identity Lemma \ref{magic polynomials} part 3) gives us the non-vanishing of the first $2 \times 2$ minor, which proves the claimed independence. So by applying to (\ref{eLin}) the diagonal transformation of dividing the first coordinate by ${u_1}^{2M + 1}$ and the second by ${u_2}^{2M+1}$, we deduce that the two vectors
\[
(v_1{u_1}^{-2M - 1}, v_2{u_2}^{-2M - 1}, -1)
\]
and
\[
(W_{M}(u_2, -1), W_M(-1, u_1), -W_M(u_1, u_2)) =: (w_1, w_2, w_3)
\]
are linearly independent. So by Lemma \ref{difference between solutions} we get that 
\[
(2M + 1)H_K(u) \leq H_K(vu^{-2M - 1}) + H_K^{\text{hom}}(w_1, w_2, w_3)
\]
But now the inequality
\[
H_K^{\text{hom}}(w_1, w_2, w_3) \leq M \cdot H_K(u)
\]
follows immediately from the non-archimedean triangle inequality. So we indeed get
\[
(M + 1) H_K(u) \leq H_K(vu^{-2M - 1}).
\]
This ends the proof.
\end{proof}

Define
\[
\text{Sol}(G) := \{(u_1,u_2) \in G \setminus G^{\text{tors}} : u_1 + u_2 = 1\}
\]
and
\[
\text{Prim-Sol}(G) := \{(u_1,u_2) \in G \setminus G^p : u_1 + u_2 = 1\}.
\] 
It is easily seen that $\text{Prim-Sol}(G) \subseteq \text{Sol}(G)$. Finally define
\[
S := \{v \in M_K : \text{ there is } g \in G \text{ with } v(g) \neq 0\}.
\]
Note that $S$ is a finite set and that one has an homomorphism $\varphi: G \to \mathbb{Z}^{|S|} \times {\mathbb{Z}}^{|S|} \subseteq \mathbb{R}^{|S|} \times {\mathbb{R}}^{|S|}$ defined by sending $(g_1, g_2) \in G$ to $(v(g_1) \deg v, v(g_2) \deg v)_{v \in S}$. 

Let $u, v \in \text{Sol}(G)$ be such that $\varphi(u) = \varphi(v)$. Suppose that $u \neq v$. Then Lemma \ref{difference between solutions} implies that $H_K(u) \leq 0$. Hence by Lemma \ref{lHeight} part 2) it follows that $u$ and thus $v$ are in $ G^{\text{tors}}$. This implies that the restriction of $\varphi$ to $\text{Sol}(G)$ is injective. In particular the restriction of $\varphi$ to $\text{Prim-Sol}(G)$ is injective. We now call $\mathcal{S} := \varphi(\text{Sol}(G))$ and $\mathcal{PS} := \varphi(\text{Prim-Sol}(G))$. It suffices to bound the cardinality of $\mathcal{PS}$.

Let $||\cdot||$ be the norm on $\mathbb{R}^{|S|} \times \mathbb{R}^{|S|}$ that is the average of the $||\cdot||_{1}$ norms on $\mathbb{R}^{|S|}$. More precisely, we define for $x=(x_1,x_2) \in \mathbb{R}^{|S|} \times \mathbb{R}^{|S|}$ 
\[
||x|| = \frac{1}{2} (||x_1|| + ||x_2||).
\]
We now state the most important properties of $\mathcal{S}$.

\begin{lemma}
\label{the sufficient properties to win}
The set $\mathcal{S} \subseteq \mathbb{Z}^{|S|} \times \mathbb{Z}^{|S|}$ has the following properties: \\
1) For any two distinct $u,v \in \mathcal{S}$, we have that $||u|| \leq 2||v - u||$. \\
2) For any two distinct $u,v \in \mathcal{S}$ and any positive integer $N$ such that $N<\frac{p}{3}-2$, there is $M \in \{N, N + 1\}$ such that $||u|| \leq \frac{2}{M + 1} ||v - (2M + 1)u||$.  \\
3) $p\mathcal{S} \subseteq \mathcal{S}$.
\end{lemma}

\begin{proof}
Let $x = (x_1, x_2) \in G$. By construction we have
\[
||\varphi(x)|| = H_K^{\text{hom}}(1, x_1) + H_K^{\text{hom}}(1, x_2).
\]
Note the basic inequalities
\[
H_K^{\text{hom}}(x_1, x_2) \leq H_K^{\text{hom}}(1, x_1) + H_K^{\text{hom}}(1, x_2) \leq 2H_K^{\text{hom}}(x_1, x_2).
\]
It is now clear that Lemma \ref{difference between solutions} implies part 1) and Lemma \ref{difference between solutions and multiples of solutions} implies part 2). Finally, part 3) is due to the action of the Frobenius operator.
\end{proof}


Denote by $V$ the real span of $\varphi(G)$. Then $V$ is an $r$-dimensional vector space over $\mathbb{R}$. We will keep writing $||\cdot||$ for the restriction of $||\cdot||$ to $V$. We have the following lemma.

\begin{lemma}
Given a positive real number $\theta$, one can find a set $\mathcal{E} \subseteq \{x \in V: ||x||=1\}$ satisfying \\
1) $|\mathcal{E}| \leq (1 + \frac{2}{\theta})^r$, \\
2) for all $0 \neq u \in V$ there exists $e \in \mathcal{E}$ satisfying $||\frac{u}{||u||} - e|| \leq \theta$.
\end{lemma}

\begin{proof}
See Lemma 6.3.4 in \cite{EG}.
\end{proof}

Let $\theta \in (0,\frac{1}{9})$ be a parameter and fix a corresponding choice of a set $\mathcal{E}$ satisfying the above properties. Given $e \in \mathcal{E}$, we define 
\[
\mathcal{S}_{e} := \left\{x \in \mathcal{S} : \left|\left|\frac{x}{||x||}-e\right|\right| \leq \theta \right\}, \ \mathcal{PS}_e:={\mathcal{S}}_e \cap \mathcal{PS}. 
\]
Fix $e \in \mathcal{E}$. We proceed to bound $|{\mathcal{PS}}_e|$. We start by deducing a so-called gap principle from part 1) of Lemma \ref{the sufficient properties to win}.

\begin{lemma}
\label{gap principle}
Let $u_1,u_2$ be distinct elements of $S_{e}$, with $||u_2|| \geq ||u_1||$. Then $||u_2|| \geq \frac{3 - \theta}{2 + \theta}||u_1||$.
\end{lemma}

\begin{proof}
Write $\lambda_i:=||u_i||$ for $i=1,2$. Then we have $u_i = \lambda_i e + u'_i$ where $||u'_i|| \leq \theta \lambda_i$, by definition of ${\mathcal{S}}_e$. Part 1) of Lemma \ref{the sufficient properties to win} gives
\[
\lambda_1 \leq 2||(\lambda_2 - \lambda_1)e + (u'_2 - u'_1)|| \leq 2(\lambda_2 - \lambda_1) + \theta(\lambda_2 + \lambda_1),
\]
and after dividing by $\lambda_1$ we get that 
\[
1 \leq  2\left(\frac{\lambda_2}{\lambda_1} - 1\right) + \theta \left(\frac{\lambda_2}{\lambda_1} + 1\right).
\]
This can be rewritten as $\frac{3 - \theta}{2 + \theta} \leq \frac{\lambda_2}{\lambda_1}$.
\end{proof}

From part 2) of Lemma \ref{the sufficient properties to win} we can deduce the following crucial Lemma.

\begin{lemma} 
\label{anti-gap principle}
Let $u_1,u_2$ be distinct elements of ${\mathcal{S}}_{e}$. Suppose that $\frac{||u_2||}{||u_1||}<\frac{2}{3}p-3$. Then $\frac{||u_2||}{||u_1||} \leq \frac{10}{\theta}$.
\end{lemma}

\begin{proof}
We follow the proof of Lemma 6.4.9 of \cite{EG} part (ii) with a few modifications. For completeness we write out the full proof.

Again denote by $\lambda_i = ||u_i||$ for $i=1, 2$, and by $u'_i = u_i - \lambda_i e$. Assume that $\lambda_2 \geq \frac{10}{\theta}\lambda_1$. Let $N$ be the positive integer with $2N + 1 \leq \frac{\lambda_2}{\lambda_1} < 2N + 3$. Then $2N + 1 < \frac{2}{3}p - 3$ and hence $N<\frac{p}{3}-2$. Applying part 2) of Lemma \ref{the sufficient properties to win} gives an integer $M \in \{N,N+1\}$ satisfying
\[
\lambda_1 \leq \frac{2}{M + 1}||(\lambda_2 - (2M+1)\lambda_1) e + u'_2 - (2M + 1)u'_1||.
\]
Furthermore, we have that
\[
|\lambda_2 - (2M + 1)\lambda_1| \leq 2\lambda_1
\]
and $M > \frac{4}{\theta}$ from the assumption $\lambda_2 \geq \frac{10}{\theta}\lambda_1$. Hence
\begin{align*}
\lambda_1 &\leq \frac{2}{M + 1}||(\lambda_2 - (2M+1)\lambda_1) e + u'_2 - (2M + 1)u'_1|| \leq \frac{2}{M + 1}(2\lambda_1 +\lambda_2\theta + (2M + 1)\lambda_1\theta) \\
&\leq \frac{2}{M + 1}(2 + (4M + 4)\theta)\lambda_1 = \left(\frac{4}{M + 1} + 8\theta\right)\lambda_1 < 9\theta\lambda_1.
\end{align*}
It follows that $\lambda_1<\frac{1}{1-9\theta}$. Now observe that for any non-negative integer $h$ the elements $p^hu_1,p^hu_2$ in ${\mathcal{S}}_e$ satisfy all the assumptions made so far. We conclude that also $p^h\lambda_1<\frac{1}{1-9\theta}$ for every non-negative integer $h$, which implies that $||u_1||=0$. This contradicts the fact that $u_1 \in {\mathcal{S}}_e$, completing the proof.
\end{proof}

Assume without loss of generality that $\mathcal{PS}_{e}$ is not empty, and fix a choice of $u_0 \in \mathcal{PS}_{e}$ with $||u_0||$ minimal. For any $u \in {\mathcal{PS}}_{e}$, denote by $k(u)$ the smallest non-negative integer such that $\frac{||u||}{p^{k(u)} ||u_0||} < p$ and denote $\lambda(u) := \frac{||u||}{p^{k(u)} ||u_0||}$.

We define $\mathcal{PS}_{e}(1) := \{u \in {\mathcal{PS}}_{e} : \lambda(u) \leq \sqrt{p}\}$ and ${\mathcal{PS}}_{e}(2) := \{u \in {\mathcal{PS}}_{e} : \lambda(u) > \sqrt{p}\}$. Since we may assume $p > 7$ by Corollary \ref{boundpr2}, we have $\frac{2p}{3} - 3 > \sqrt{p}$.

\begin{lemma}
\label{gap and anti gap}
1) Let $i \in \{1, 2\}$ and let $u_1, u_2$ be distinct elements of ${\mathcal{PS}}_{e}(i)$ with $\lambda(u_2) \geq \lambda(u_1)$. Then $\lambda(u_2) \geq \frac{3 - \theta}{2+\theta}\lambda(u_1)$ and $\lambda(u_2) \leq \frac{10}{\theta}\lambda(u_1)$. \\
2) $\lambda({\mathcal{PS}}_{e}(2)) \subseteq [\frac{\theta p}{10},p)$. \\
3) $\lambda$ is an injective map on ${\mathcal{PS}}_{e}$.
\end{lemma}

\begin{proof}
1) We apply Lemma \ref{gap principle} and \ref{anti-gap principle} to the pair $(p^{k(u_2)-k(u_1)}u_1,u_2)$ if $k(u_2) \geq k(u_1)$ and to the pair $(u_1, p^{k(u_1)-k(u_2)}u_2)$ otherwise. We stress that these elements are indeed distinct, since $u_1,u_2 \in \mathcal{PS}$. This gives the desired result. \\
2) This follows from Lemma \ref{anti-gap principle} applied to the pair $(u_1, p^{k(u_1) + 1}u_0)$ for each $u_1$ in ${\mathcal{PS}}_{e}(2)$. \\
3) Use part 1) and the fact that $\frac{3 - \theta}{2 + \theta} > 1$ for $\theta \in (0, \frac{1}{9})$.
\end{proof} 

By part 3) of Lemma \ref{gap and anti gap} it suffices to bound $|\lambda({\mathcal{PS}}_{e})|$. By part 1) and 2) of Lemma \ref{gap and anti gap} it will follow that we can bound $|\lambda({\mathcal{PS}}_{e})|$ purely in terms of $\theta$: thus collecting all the bounds for $e$ varying in $\mathcal{E}$ we obtain a bound depending only on $r$. We now give all the details.

For any $\theta \in (0, \frac{1}{9})$ we have
\[
\frac{3 - \theta}{2 + \theta} > \frac{26}{19}.
\]
Then we find that $|\lambda({\mathcal{PS}}_{e}(1))|$ is at most the biggest $n$ such that 
\[
\left(\frac{26}{19}\right)^{n-1} \leq \frac{10}{\theta}
\]
and similarly for $|\lambda({\mathcal{PS}}_{e}(2))|$. We conclude that
\[
|{\mathcal{PS}}_{e}| \leq 2 + 2\frac{\text{log}(\frac{10}{\theta})}{\text{log}(\frac{26}{19})}.
\]
Multiplying by $|\mathcal{E}|$ gives that for every $\theta \in (0,\frac{1}{9})$
\[
|{\mathcal{PS}}| \leq 2\left(1 + \frac{\text{log}(\frac{10}{\theta})}{\text{log}(\frac{26}{19})}\right)\left(1+\frac{2}{\theta}\right)^r.
\] 
So letting $\theta$ increase to $\frac{1}{9}$ we obtain
\[
|{\mathcal{PS}}| \leq 2\left(1 + \frac{\text{log}(90)}{\text{log}(\frac{26}{19})}\right) 19^r < 31 \cdot 19^r.
\]                   
This completes the proof of Theorem \ref{tXY}.

\section{Proof of Theorem \ref{tVoloch}}
First suppose that $G$ and $K$ are finitely generated. Before we can start with the proof of Theorem \ref{tVoloch}, we will rephrase Theorem \ref{tXY}. Recall that we write $\mathbb{F}_q$ for the algebraic closure of $\mathbb{F}_p$ in $K$.

Then Theorem \ref{tXY} implies that there is a finite subset $S$ of $G$ with $|S| \leq 31 \cdot 19^{r}$ such that any solution of
\[
x + y = 1, (x, y) \in G
\]
with $x \not \in \mathbb{F}_q$ and $y \not \in \mathbb{F}_q$ satisfies $(x, y) = (\gamma, \delta)^{p^t}$ for some $t \in \mathbb{Z}_{\geq 0}$ and $(\gamma, \delta) \in S$.

Now let $(x, y) \in G$ be a solution to
\[
ax + by = 1.
\]
If $ax \in \mathbb{F}_q$ or $by \in \mathbb{F}_q$, it follows that $ax \in \mathbb{F}_q$ and $by \in \mathbb{F}_q$, which implies that $(a, b)^{q - 1} \in G$. Hence Theorem \ref{tVoloch} holds.

So from now on we may assume that $ax \not \in \mathbb{F}_q$ and $by \not \in \mathbb{F}_q$. Define $G'$ to be the group generated by $G$ and the tuple $(a, b)$. Then the rank of $G'$ is at most $r + 1$. Let $S \subseteq G'$ be as above, so $|S| \leq 31 \cdot 19^{r + 1}$. We can write
\[
(ax, by) = (\gamma, \delta)^{p^t}
\]
with $t \in \mathbb{Z}_{\geq 0}$ and $(\gamma, \delta) \in S$. Since
$S \subseteq G'$, we can write
\[
(\gamma, \delta) = (a^k x_0, b^k y_0)
\]
with $k \in \mathbb{Z}$ and $(x_0, y_0) \in G$. This means that
\[
(ax, by) = (a^k x_0, b^k y_0)^{p^t},
\]
which implies $(a, b)^{kp^t - 1} \in G$. If $kp^t - 1$ is co-prime to $p$, we conclude again that Theorem \ref{tVoloch} holds. But $p$ can only divide $kp^t - 1$ if $t = 0$. Then we find immediately that there are at most $|S| \leq 31 \cdot 19^{r + 1}$ solutions as desired.

We still need to deal with the case that $K$ is an arbitrary field of characteristic $p$ and $G$ is a subgroup of $K^\ast \times K^\ast$ with $\text{dim}_\mathbb{Q} \ G \otimes_\mathbb{Z} \mathbb{Q} = r$ finite. Suppose that $ax + by = 1$ has more than $31 \cdot 19^{r + 1}$ solutions $(x, y) \in G$. Then we can replace $G$ by a finitely generated subgroup of $G$ with the same property. We can also replace $K$ by a subfield, finitely generated over its prime field, containing the coordinates of the new $G$ and $a, b$. This gives the desired contradiction.

\section{Acknowledgements}
We are grateful to Julian Lyczak for explaining us how identities as in Lemma \ref{magic identity} follow from basic properties of hypergeometric functions. Many thanks go to Jan-Hendrik Evertse for providing us with this nice problem, his help throughout and the proofreading.

\end{document}